\documentclass{article}
\usepackage{amsmath,amsthm,amssymb}
\usepackage{amsfonts,mathtools}
\usepackage{tcs-macros}
\usepackage[letterpaper,margin=1in]{geometry}

\DeclarePairedDelimiterX{\divvert}[2]{(}{)}{%
  {#1}\;\delimsize\|\;{#2}%
}
\DeclarePairedDelimiterX{\divc}[2]{(}{)}{%
  {#1}\delimsize;{#2}%
}
\newcommand{\KL}{D_{\mathrm{KL}}\divvert}

\DeclarePairedDelimiter{\paren}{(}{)}
\DeclarePairedDelimiter{\sqpar}{[}{]}
\newcommand{\ent}{\mathrm{H}\paren}
\newcommand{\info}{\mathrm{I}\paren}
\newcommand{\Ruz}{\mathrm{d}\sqpar}
\newcommand{\pot}{\phi\sqpar}

\usepackage{graphicx} 

\title{Improved Exponent for Marton's Conjecture in $\mathbb{F}_2^n$}
\date{April 15, 2024}
\author{Jyun-Jie Liao\thanks{Supported by Eshan Chattopadhyay's NSF CAREER award 2045576.}\\
Cornell University\\
\texttt{jjliao@cs.cornell.edu}}

\begin{document}

\maketitle
\begin{abstract}
A conjecture of Marton, widely known as the polynomial Freiman-Ruzsa conjecture, was recently proved by Gowers, Green, Manners and Tao for any bounded-torsion Abelian group $G$. In this paper we show a few simple modifications that improve their bound in $G=\bbF_2^n$. Specifically, for $G=\bbF_2^n$, they proved that any set $A\subseteq G$ with $\abs{A+A}\le K\abs{A}$ can be covered by at most $2K^C$ cosets of a subgroup $H$ of $G$ of cardinality at most $\abs{A}$, with $C=12$. In this paper we prove the same statement for $C=9$.
\end{abstract}

\section{Introduction}

A conjecture of Katalin Marton (see \cite{PFR}), widely known as the polynomial Freiman–Ruzsa conjecture, was recently proved by Gowers, Green, Manners and Tao~\cite{GGMT23,GGMT24} for any bounded-torsion Abelian group $G$. (See \cite{GGMT23} for discussions about its importance and prior progress on this conjecture.) Specifically, for the case of $G=\bbF_2^n$, Marton's conjecture is as follows.
\begin{conjecture}\label{conj:Marton}
Let $A\subseteq \bbF_2^n$ be a set with $\abs{A+A}\le K\abs{A}$. Then there exists a subspace $V\subseteq \bbF_2^n$ of size at most $\abs{A}$ such that $A$ can be covered by at most $2K^C$ translates of $V$, for some constant $C$.
\end{conjecture}

In \cite{GGMT23}, it was proved that \Cref{conj:Marton} is true for $C=12$.\footnote{More precisely, their argument shows that $C=7+\sqrt{17}=11.1231...$.} In this paper, we show a few simple refinements to the proof in \cite{GGMT23} and prove that \Cref{conj:Marton} is true for $C=9$.
\noindent
\begin{theorem}\label{theorem:Marton}
\Cref{conj:Marton} is true for $C=9$.
\end{theorem}

In more details, recall that the entropic Ruzsa distance between two random variables $\bX,\bY$, denoted by $\Ruz{\bX;\bY}$, is defined as $$\Ruz{\bX;\bY}:=\ent{\bX'-\bY'}-\frac{1}{2}(\ent{\bX}+\ent{\bY}),$$
where $\bX',\bY'$ are independent copies of $\bX,\bY$.\footnote{Note that $\Ruz{\cdot;\cdot}$ can be viewed as a function on individual distributions of $\bX$ and $\bY$. Therefore, even if $(\bX,\bY)$ are correlated, the value $\Ruz{\bX;\bY}$ is defined as if they are independent.} It was observed in \cite{GMT23} that \Cref{conj:Marton} is equivalent to the following ``entropic variant", up to some loss in the constant.
\begin{conjecture}\label{conj:Marton-entropic}
For every pair of random variables $\bA,\bB$ on $\bbF_2^n$, there exists a subspace $V$ such that $\Ruz{\bA;\bU_V}+\Ruz{\bB;\bU_V}\le C' \Ruz{\bA;\bB}$ for some constant $C'$, where $\bU_V$ is uniform on $V$.
\end{conjecture}
\noindent
Our first result is an improved bound for the entropic version, from $C'=11$ in \cite{GGMT23} to $C'=10$.
\begin{theorem}\label{theorem:Marton-entropic}
\Cref{conj:Marton-entropic} is true for $C'=10$.
\end{theorem}
\noindent
Based on the equivalence established in \cite{GMT23} which shows that $C\le C'+1$, \Cref{theorem:Marton-entropic} already implies that \Cref{conj:Marton} is true with $C=11$. To further establish the $C=9$ bound, we show that a simple modification to the proof of \Cref{theorem:Marton-entropic} can directly imply a better bound for the covering variant.

\paragraph{Notation.} All logarithms in this paper will be of base $2$. We use $\ent{\cdot},\info{\cdot},D_{\mathrm{KL}}(\cdot)$ to denote Shannon entropy, mutual information and Kullback–Leibler divergence respectively. (We review their definitions and some basic properties in \Cref{appendix:definition}.) For any random variable $\bX$, we abuse notation and also use $\bX$ to denote its distribution. All random variables are in bold font. For random variable $\bX$ (and its distribution) we use $\Supp(\bX)$ to denote its support. For any set $S$, $\bU_S$ denotes a random variable that is uniform over $S$. 

For any real-valued function $f$ on random variables or distributions, and any pair of joint random variables $(\bX,\bZ)$, we use $f(\bX\mid\bZ)$ to denote $\ex[z\sim\bZ]{f(\bX|_{\bZ=z})}$. Similarly for a function $g(\cdot;\cdot)$ on two random variable we write $g((\bX;\bY)\mid\bZ)=\ex[z\sim\bZ]{g(\bX|_{\bZ=z};\bY|_{\bZ=z})}$. In addition, we sometimes write $g(\bX\mid \bZ ;\bY\mid \bW)=\ex[z\sim\bZ,w\sim\bW]{g(\bX|_{\bZ=z};\bY|_{\bW=w})}$ when $(\bX,\bZ)$ is independent of $(\bY,\bW)$.

\section{Recap of the Original Proof}
In this section we briefly recap the proof of entropic Marton's conjecture in \cite{GGMT23}. The outline of the proof is as follows. Fix two starting random variables $\bA$ and $\bB$ on $\bbF_2^n$, and let $\tau_A(\cdot),\tau_B(\cdot)$ be some functions over random variables in $\bbF_2^n$, so that $\tau_A(\bX),\tau_B(\bX)$ denote certain ``divergence" of a random variable $\bX$ from $\bA$ and $\bB$, and that $\tau_A,\tau_B$ are both invariant under translation. To prove the entropic version (\Cref{conj:Marton-entropic}) we should choose $\tau_A(\bX):=\Ruz{\bB;\bX}$ and $\tau_B(\bX):=\Ruz{\bA;\bX}$. The overall strategy is to start with $\bX=\bA,\bY=\bB$, and keep updating $\bX,\bY$ so that $\Ruz{\bX;\bY}$ always becomes strictly smaller than in the previous round. At the same time, we want to make sure that $\tau_A(\bX)$ and $\tau_B(\bY)$ do not grow too much in every round. Recall that $\Ruz{\bX;\bY}$ is non-negative, and is $0$ if and only if there exists a subspace $V$ such that $\bX$ and $\bY$ are both uniform over some translates of $V$. (See, e.g., \cite[Theorem 1.11]{Tao10}.) Therefore, when $\Ruz{\bX;\bY}$ eventually decreases to $0$, if at the same time we can show that $\tau_A(\bX)=\tau_A(\bU_V)$ and $\tau_B(\bY)=\tau_B(\bU_V)$ are not too large, then we are done. 

To update $\bX$ and $\bY$, the primary choices are either ``sums" or ``fibres". In more details, let $\bX_1,\bX_2$ be independent copies of $\bX$, and $\bY_1,\bY_2$ be independent copies of $\bY$ (that are also independent of $\bX_1,\bX_2$). Consider the following random variables:
\begin{align*}
\bT&:=\bX_1+\bY_1, & \obT&:=\bX_2+\bY_2,\\
\bV&:=\bX_1+\bY_2, & \obV&:=\bX_2+\bY_1,\\
\bW&:=\bX_1+\bX_2, & \obW&:=\bY_1+\bY_2.\\
\end{align*}
The typical next choice of $(\bX,\bY)$, denote by $(\bX',\bY')$, is either obtained by taking the sums: $(\bT,\obT)$ or $(\bW,\obW)$; or fibres: $(\bX_1|\bT, \bY_2|\obT)$ or $(\bX_1|\bW, \bY_1|\obW)$.\footnote{For joint random variables $(\bR,\bZ)$ we use $\bR|\bZ$ to denote $\bR|_{\bZ=z}$ for $z$ randomly sampled from $\bZ$. $z$ will be fixed in the end.} The following corollary of fibring lemma~\cite{GMT23,GGMT23} guarantees that at least one choice is not worse than the original $(\bX,\bY)$.
\begin{lemma}[{\cite[Corollary 4.2]{GGMT23}}]\label{lemma:fibring}
For independent random variables $\bR_1,\bR_2,\bR_3,\bR_4$ on $\bbF_2^n$,
\begin{align*}
\Ruz{\bR_1 ; \bR_3}+\Ruz{\bR_2 ; \bR_4}&=\Ruz{\bR_1+\bR_2 ; \bR_3+\bR_4}\\
&+\Ruz{\bR_1\mid\bR_1+\bR_2 \,\,; \,\,\bR_3\mid\bR_3+\bR_4}\\
&+\info{\bR_1+\bR_2:\bR_1+\bR_3\mid\bR_1+\bR_2+\bR_3+\bR_4}.
\end{align*}
\end{lemma}
\noindent
Now define $\bS:=\bX_1+\bX_2+\bY_1+\bY_2$, $I_1:=\info{\bT:\bV\mid\bS}$, and $I_2:=\info{\bT:\bW\mid\bS}=\info{\bV:\bW\mid\bS}$. One can obtain the following equations by substituting $(\bR_1,\bR_2,\bR_3,\bR_4)$ in \Cref{lemma:fibring} with $(\bX_1,\bY_1,\bY_2,\bX_2)$ and $(\bX_1,\bX_2,\bY_1,\bY_2)$:

\begin{align}
\Ruz{\bT;\obT}+\Ruz{\bX_1|\bT;\bY_2|\obT}&=2\Ruz{\bX;\bY}-I_1\label{eq:fib1}\\
\Ruz{\bW;\obW}+\Ruz{\bX_1|\bW;\bY_1|\obW}&=2\Ruz{\bX;\bY}-I_2\label{eq:fib2}
\end{align}
By (\ref{eq:fib1}) and (\ref{eq:fib2}), for at least one choice of $(\bX',\bY')$ from the sums or fibres, we can get $\Ruz{\bX';\bY'}\le\Ruz{\bX;\bY}$. 
We also record the identity obtained by substituting with $(\bX_1,\bY_1,\bX_2,\bY_2)$:
$$\Ruz{\bT;\obT}+\Ruz{\bX_1|\bT;\bX_2|\obT}=\Ruz{\bX;\bX}+\Ruz{\bY;\bY}-I_2$$
Together with (\ref{eq:fib1}) and the fact that $\Ruz{\bX_1|\bT;\bY_2|\obT}=\Ruz{\bX_1|\bT;\bX_2|\obT}$ we get 
\begin{align}\label{eq:xxyy}
\Ruz{\bX;\bX}+\Ruz{\bY;\bY}=2\Ruz{\bX;\bY}+(I_2-I_1),
\end{align}
which helps simplify some computation later. 

Note that when picking only from sums or fibres, there is a chance that we make no progress, or the progress is too small to compensate the increase in $\tau_A,\tau_B$. However, in this case we get that $I_1,I_2$ are both very small, which means $\bT,\bV,\bW$ are almost pairwise independent when conditioned on $\bS$. This is called the ``endgame" in \cite{GGMT23}. In the endgame, when conditioned on of $\bS$ we roughly have that $\Ruz{\bT;\bV}$ is very close to $\ent{\bT+\bV}-\frac{1}{2}(\ent{\bT}+\ent{\bV})=\ent{\bW}-\frac{1}{2}(\ent{\bT}+\ent{\bV})$, and similarly $\Ruz{\bV;\bW}$ is close to $\ent{\bT}-\frac{1}{2}(\ent{\bV}+\ent{\bW})$ and $\Ruz{\bT;\bW}$ is close to $\ent{\bU}-\frac{1}{2}(\ent{\bT}+\ent{\bW})$. Because the sum of these approximations is exactly $0$, the best choice of pair from $\{\bT|_{\bS=s},\bV|_{\bS=s},\bW|_{\bS=s}\}$ should have very small entropic Ruzsa distance. This intuition can be formalized by the entropic Balog–Szemer\'{e}di–Gowers lemma~\cite{Tao10,GGMT23} as follows.
\begin{lemma}[{\cite[Lemma A.2]{GGMT23}}]
Let $G$ be an additive group and $(\bR_1,\bR_1)$ be random variables on $G^2$. Then 
$$\Ruz{(\bR_1;\bR_2)\mid \bR_1+\bR_2}\le 3\info{\bR_1:\bR_2}+2\ent{\bR_1+\bR_2}-\ent{\bR_1}-\ent{\bR_2}.$$
\end{lemma}
\noindent
If we substitute $(\bR_1,\bR_2)$ with $(\bT,\bV),(\bV,\bW),(\bW,\bT)$ conditioned on $\bS$ and take their sum, we get the following inequality:
\begin{equation}\label{eq:endgame}
\Ruz{(\bT;\bV)\mid \obW,\bS}+\Ruz{(\bV;\bW)\mid \obT,\bS}+\Ruz{(\bW;\bT)\mid\obV,\bS}\le 3I_1+6I_2,
\end{equation}
which implies that the best pair from $\{\bT|\bS,\bV|\bS,\bW|\bS\}$ has entropic Ruzsa distance at most $I_1+2I_2$.\footnote{In fact, every pair has entropic Ruzsa distance at most $I_1+2I_2$. See \cite[p.23]{GGMT23}.}

It remains to prove that $\tau_{A},\tau_{B}$ do not grow too much with each choice. For simplicity, think of $I_1$ and $I_2$ as equal. In this case it was proved in \cite{GGMT23} that $\tau_A(\bX')+\tau_B(\bY')\le \tau_A(\bX)+\tau_B(\bY)+\Ruz{\bX:\bY}$ for $(\bX',\bY')$ picked from sums or fibres. In the endgame, it was also proved that the growth of $\tau_A(\bX)+\tau_B(\bY)$ in the three possible choices only sum up to slightly larger than $6\Ruz{\bX:\bY}$. Therefore, one can either update $(\bX,\bY)$ to be a choice $(\bX',\bY')$ from sums or fibres that satisfies $\Ruz{\bX';\bY'}\le (8/9) \Ruz{\bX;\bY}$; or proceed to the endgame otherwise, where $I_1,I_2\le (2/9)\Ruz{\bX;\bY}$ and there exists $(\bX',\bY')$ such that $\Ruz{\bX';\bY'}\le (6/9)\Ruz{\bX;\bY}$. Because the growth in $\tau_A(\bX)+\tau_B(\bY)$ is at most $9$ times the decrease in $\Ruz{\bX;\bY}$, when $\Ruz{\bX;\bY}$ decreases from $\Ruz{\bA;\bB}$ to $0$, $\tau_A(\bX)+\tau_B(\bY)$ only grows by at most $9\Ruz{\bA;\bB}$ from the starting value $\tau_A(\bA)+\tau_B(\bB)=2\Ruz{\bA;\bB}$. Therefore we get the bound $\Ruz{\bA;\bU_V}+\Ruz{\bB;\bU_V}\le 11\Ruz{\bA;\bB}$ in \cite{GGMT23}.

\section{Improved Bound for Entropic Variant}
In this work, our improvement in \Cref{theorem:Marton-entropic} comes from a tighter bound for the growth of $\tau_A(\bX)+\tau_B(\bY)$ in the endgame. In the $I_1=I_2$ case, we show that the growth of $\tau_A(\bX)+\tau_B(\bY)$ for three endgame choices sum up to at most $6\Ruz{\bX:\bY}$. Therefore we can pick the threshold for endgame to be $I_1,I_2\le \Ruz{\bX;\bY}/4$ instead, and improve the constant factor to $10$ correspondingly. First we state our formal claim.
\begin{lemma}\label{lemma:main}
Let $\tau_A,\tau_B$ be any real-valued functions on distributions over $\bbF_2^n$. Suppose that every $\tau\in\{\tau_A,\tau_B\}$ satisfies the following properties for any independent random variables $\bX,\bY$ on $\bbF_2^n$, and any $\bZ$ correlated with $\bX$:
\begin{itemize}
\item 
$\tau\paren{\bX+\bY}\le \tau\paren{\bX}+\frac{1}{2}\paren{\ent{\bX+\bY}-\ent{\bX}}$
\item 
$\tau\paren{\bX\mid\bZ}\le  \tau\paren{\bX}+\frac{1}{2}\paren{\ent{\bX}-\ent{\bX\mid\bZ}}$
\item
$\tau(\bX+s)=\tau(\bX)$ for every $s\in\bbF_2^n$
\item
$\tau$ is continuous
\end{itemize}
Then there exists a subspace $V$ that $\tau_A(\bU_V)+\tau_B(\bU_V)\le \tau_A(\bX)+\tau_B(\bY)+8\Ruz{\bX;\bY}$ for any random variables $\bX,\bY$.
\end{lemma}
\noindent 
Note that \Cref{lemma:main} directly implies \Cref{theorem:Marton-entropic}:
\begin{proof}[Proof of \Cref{theorem:Marton-entropic}]
If both $\Ruz{\bA;\cdot}$ and $\Ruz{\bB;\cdot}$ satisfy all the conditions in \Cref{lemma:main} then we can substitute $\tau_A(\cdot),\tau_B(\cdot)$ with $\Ruz{\bB;\cdot},\Ruz{\bA;\cdot}$ respectively and conclude that there exists a subspace $V$ s.t. $\Ruz{\bU_V;\bA}+\Ruz{\bU_V,\bB}\le 10\Ruz{\bA;\bB}$. The third and fourth condition are trivially true. The first condition was proved in \cite[Lemma 5.2]{GGMT23}. For the second condition, observe that for every joint random variables $(\bX,\bZ)$ and any independent $\bR$ we have 
$$\Ruz{\bR;\bX\mid\bZ}-\Ruz{\bR;\bX}=\ent{\bR+\bX\mid\bZ}-\ent{\bR+\bX}+\frac{1}{2}(\ent{\bX}-\ent{\bX\mid\bZ})\le \frac{1}{2}(\ent{\bX}-\ent{\bX\mid\bZ}).$$

\end{proof}

As in \cite{GGMT23}, to prove \Cref{lemma:main} it suffices to prove the following lemma. 
\begin{lemma}\label{lemma:decrease}
Let $\eta<1/8$ be a real parameter, $\tau_A,\tau_B$ be any functions that satisfy the conditions in \Cref{lemma:main}, and define $\pot{\bX;\bY}:=\Ruz{\bX;\bY}+\eta(\tau_A(\bX)+\tau_B(\bY))$. Then for any $\bbF_2^n$-valued random variables $\bX,\bY$ such that $\Ruz{\bX;\bY}>0$, there exist $\bbF_2^n$-valued random variables $\bX',\bY'$ such that
$\pot{\bX';\bY'}<\pot{\bX;\bY}$.
\end{lemma}
\noindent
To see why the lemma above implies \Cref{lemma:main}, let $(\bX^*,\bY^*)$ be a minimizer of $\phi$, which exists because $\phi$ is continuous and distributions over $\bbF_2^n$ are compact. By \Cref{lemma:decrease} it must be the case that $\Ruz{\bX^*;\bY^*}=0$, which implies a subspace $V$ such that $\tau_A(\bU_V)=\tau_A(\bX^*)$ and $\tau_B(\bU_V)=\tau_A(\bY^*)$. Because $(\bX^*,\bY^*)$ minimizes $\phi$, we can conclude that $\tau_A(\bU_V)+\tau_B(\bU_V)\le \tau_A(\bX)+\tau_B(\bY)+\frac{1}{\eta}\Ruz{\bX;\bY}\le \tau_A(\bX)+\tau_B(\bY)+8\Ruz{\bX;\bY}$ for any $(\bX,\bY)$.

\begin{proof}[Proof of \Cref{lemma:decrease}]
First we define random variables $\bX_1,\bX_2,\bY_2,\bY_2,\bT,\bV,\bW,\obT,\obV,\obW,\bS$ and real values $I_1,I_2$ based on $\bX,\bY$ as in the previous section. In addition, let $\tau_0=\tau_A(\bX)+\tau_B(\bY)$, and further define the following values:
\begin{align*}
\tau^+_T&:=\tau_A(\bT)+\tau_B(\obT), & \tau^-_T&:=\tau_A(\bX_1|\bT)+\tau_B(\bY_2|\obT),\\
\tau^+_W&:=\tau_A(\bW)+\tau_B(\obW), & \tau^-_W&:=\tau_A(\bX_1|\bW)+\tau_B(\bY_1|\obW).\\
\end{align*}
Each of these values can be bounded by roughly $\tau_0+\Ruz{\bX;\bY}$ as in the following inequalities. (For $\tau^+$ we use the first condition in \Cref{lemma:main}. For $\tau^-$ we use the second condition in \Cref{lemma:main} and the fact that $\ent{\bR_1}-\ent{\bR_1|\bR_1+\bR_2}=\ent{\bR_1+\bR_2}-\ent{\bR_2}$ for independent $\bR_1,\bR_2$. This identity will also be used later when bounding $\tau$ in the endgame.)
\begin{align}
\tau_T^+
&\le \tau_A(\bX)+\frac{1}{2}(\ent{\bX_1+\bY_1}-\ent{\bX})+\tau_B(\bY)+\frac{1}{2}(\ent{\bX_2+\bY_2}-\ent{\bY})\nonumber\\
&= \tau_0+\Ruz{\bX;\bY}\label{eq:T+}\\
\tau_W^+
&\le \tau_A(\bX)+\frac{1}{2}(\ent{\bX_1+\bX_2}-\ent{\bX})+\tau_B(\bY)+\frac{1}{2}(\ent{\bY_1+\bY_2}-\ent{\bY})\nonumber\\
&= \tau_0+\frac{1}{2}\Ruz{\bX;\bX}+\frac{1}{2}\Ruz{\bY;\bY}\nonumber\\
&= \tau_0+\Ruz{\bX;\bY}+\frac{1}{2}(I_2-I_1)  \textrm{    (by (\ref{eq:xxyy}))} \label{eq:W+}\\
\tau_T^-
&\le \tau_A(\bX)+\frac{1}{2}(\ent{\bX_1+\bY_1}-\ent{\bY})+\tau_B(\bY)+\frac{1}{2}(\ent{\bX_2+\bY_2}-\ent{\bX})\nonumber\\
&= \tau_0+\Ruz{\bX;\bY}\label{eq:T-}\\
\tau_W^-
&\le \tau_A(\bX)+\frac{1}{2}(\ent{\bX_1+\bX_2}-\ent{\bX})+\tau_B(\bY)+\frac{1}{2}(\ent{\bY_1+\bY_2}-\ent{\bY})\nonumber\\
&= \tau_0+\frac{1}{2}\Ruz{\bX;\bX}+\frac{1}{2}\Ruz{\bY;\bY}\nonumber\\
&= \tau_0+\Ruz{\bX;\bY}+\frac{1}{2}(I_2-I_1)\label{eq:W-}
\end{align}

If either $\Ruz{\bT;\obT} + \eta\tau^+< \Ruz{\bX;\bY}+\eta \tau_0$ or $\Ruz{\bW;\obW}+\eta\tau^+_W< \Ruz{\bX;\bY}+\eta\tau_0$, then the claim holds for $(\bX',\bY')=(\bT,\obT)$ or $(\bX',\bY')=(\bW,\bW')$. If $\Ruz{\bX_1|\bT;\bY_2|\obT}+ \eta \tau^0_T< \Ruz{\bX;\bY}+\eta\tau_0$ then the claim holds for $(\bX',\bY')=(\bX_1|_{\bT=t},\bY_2|_{\obT=t'})$, where $t,t'$ minimizes $\pot{\bX_1|_{\bT=t};\bY_2|_{\obT=t'}}$.\footnote{Note that $\pot{\bX_1|\bT;\bY_2|\bT'}:=\ex[t\sim\bT,t'\sim\bT']{\pot{\bX_1|_{\bT=t};\bY_2|_{\obT=t'}}}=\Ruz{\bX_1|\bT;\bY_2|\obT}+\eta\tau^-_T$ by linearity of expectation.} Similarly if $\Ruz{\bX_1|\bW;\bY_1|\obW}+\eta\tau_0< \Ruz{\bX;\bY}+\eta\tau_W^-$ we are also done. If none of these conditions hold, then we are in the endgame and 
\begin{align}
\Ruz{\bT;\obT}&\ge\Ruz{\bX;\bY}-\eta(\tau^+_T-\tau_0),& \Ruz{\bX_1|\bT;\bY_2|\obT}&\ge \Ruz{\bX;\bY}-\eta(\tau^-_T-\tau_0),\label{eq:T-endgame}\\
\Ruz{\bW;\obW}&\ge\Ruz{\bX;\bY}-\eta(\tau^+_W-\tau_0),& \Ruz{\bX_1|\bW;\bY_1|\obW}&\ge \Ruz{\bX;\bY}-\eta(\tau^-_W-\tau_0).\label{eq:W-endgame}
\end{align}
From (\ref{eq:fib1}), (\ref{eq:T+}), (\ref{eq:T-}), (\ref{eq:T-endgame}) we have 
$I_1\le 2\eta \Ruz{\bX;\bY}$,
and from (\ref{eq:fib2}), (\ref{eq:W+}), (\ref{eq:W-}), (\ref{eq:W-endgame}) we have $I_2\le 2\eta \Ruz{\bX;\bY}+\eta(I_2-I_1)$, or equivalently 
\begin{align}\label{eq:I2}
I_2-2\eta\Ruz{\bX;\bY}\le \frac{\eta}{1-\eta} (2\eta\Ruz{\bX;\bY}-I_1).
\end{align}

Next we bound $\tau_A+\tau_B$ in the endgame. This is where our improvement comes from. Note that for any $\tau\in\{\tau_A,\tau_B\}$ we have $\tau(\bT\mid\obW,\bS)=\tau(\bV\mid\obW,\bS)$, $\tau(\bT\mid\obV,\bS)=\tau(\bV\mid\obT,\bS)$, $\tau(\bW\mid\obT,\bS)=\tau(\bW\mid\obV,\bS)$, so it suffices to bound $\tau(\bT\mid\obW,\bS)$, $\tau(\bT\mid\obV,\bS)$ and $\tau(\bW\mid\obT,\bS)$. In addition, because $\tau$ is invariant under translation, $\tau(\obT\mid\bS)=\tau(\obT+\bS\mid\bS)=\tau(\bT\mid\bS)$. By the second condition in \Cref{lemma:main} we get
\begin{align*}
\tau_A(\bT\mid\obW,\bS) &\le \tau_A(\bT\mid\bS) + \frac{1}{2}\info{\bT:\obW\mid \bS} \\
                    &=\tau_A(\bT\mid\bS) + \frac{1}{2}\info{\bT:\bW\mid \bS} \\
                     &\le  \tau_A(\bT) + \frac{1}{2}(\ent{\bS}-\ent{\obT}) + \frac{1}{2}I_2
\end{align*}
and
\begin{align*}
\tau_B(\bT\mid\obW,\bS) &\le \tau_B(\bT\mid\bS) + \frac{1}{2}\info{\bT:\obW\mid \bS} \\
                        &= \tau_B(\obT\mid\bS) + \frac{1}{2}\info{\bT:\bW\mid \bS} \\
                     &\le  \tau_B(\obT) + \frac{1}{2}(\ent{\bS}-\ent{\bT}) + \frac{1}{2}I_2.
\end{align*}
Take the sum of the above two inequalities, we get 
\begin{align}
\tau_A(\bT\mid\obW,\bS)+\tau_B(\bT\mid\obW,\bS) &\le  \tau_T^+ + \Ruz{\bT;\obT} + I_2.\label{eq:+1}
\end{align}
Apply a similar argument for $\tau(\bT\mid\obV,\bS)$ and $\tau(\bW\mid\obT,\bS)$ we get 
\begin{align}
\tau_A(\bT\mid\obV,\bS)+\tau_B(\bT\mid\obV,\bS) &\le  \tau_T^+ + \Ruz{\bT;\obT} + I_1,\label{eq:+2}\\
\tau_A(\bW\mid\obT,\bS)+\tau_B(\bW\mid\obT,\bS) &\le  \tau_W^+ + \Ruz{\bW;\obW} + I_2.\label{eq:+3}
\end{align}
In addition, by applying the first condition in \Cref{lemma:main}, 
\begin{align*}
\tau_A(\bT\mid\obW,\bS) &= \tau_A(\bT\mid \bW,\obW)\\
                    &\le \tau_A(\bX_1\mid\bW,\obW) + \frac{1}{2}(\ent{\bX_1+\bY_1\mid\bW,\obW}-\ent{\bX_1\mid\bW,\obW})\\
\tau_B(\bT\mid\obW,\bS) &= \tau_B(\bT\mid \bW,\obW)\\
                    &\le \tau_B(\bY_1\mid\bW,\obW) + \frac{1}{2}(\ent{\bX_1+\bY_1\mid\bW,\obW}-\ent{\bY_1\mid\bW,\obW})
\end{align*}
Take the sum of the above two inequalities, and use the fact that $(\bX_1,\bW)$ is independent of $(\bY_1,\obW)$, we can obtain
\begin{align}
\tau_A(\bT\mid\obW,\bS) + \tau_B(\bT\mid\obW,\bS) \le \tau_W^- + \Ruz{\bX_1|\bW;\bY_1;\obW}\label{eq:-1}
\end{align}
and similarly 
\begin{align}
\tau_A(\bT\mid\obV,\bS)+\tau_B(\bT\mid\obV,\bS) &\le  \tau_T^- + \Ruz{\bX_1\mid \bV; \bY_1\mid \obV} = \tau_T^- + \Ruz{\bX_1\mid \bT; \bY_2\mid \obT} \label{eq:-2},\\
\tau_A(\bW\mid\obT,\bS)+\tau_B(\bW\mid\obT,\bS) &\le  \tau_T^- + \Ruz{\bX_1\mid \bT;\bY_2\mid \obT}. \label{eq:-3}
\end{align}
Take the sum of (\ref{eq:+1}), (\ref{eq:+2}), (\ref{eq:+3}), (\ref{eq:-1}), (\ref{eq:-2}), (\ref{eq:-3}) and apply the fibring identities (\ref{eq:fib1}), (\ref{eq:fib2}) we get 
\begin{align}
&2 (\tau_A(\bT\mid\obW,\bS) + \tau_B(\bT\mid\obW,\bS)+ \tau_A(\bT\mid\obV,\bS) + \tau_B(\bT\mid\obV,\bS) + \tau_A(\bW\mid\obT,\bS)+\tau_B(\bW\mid\obT,\bS))\nonumber\\
&\le 2\tau_T^+ + 2\tau_T^- +\tau_W^++\tau_W^- + 6\Ruz{\bX;\bY} +  (I_2-I_1)\nonumber\\
&\le 6\tau_0 + 12\Ruz{\bX;\bY} + 2(I_2-I_1). \textrm{   (by (\ref{eq:T-}), (\ref{eq:T+}), (\ref{eq:W-}), (\ref{eq:W+}))}\label{eq:tau-eg}
\end{align}
Now define 
\begin{align*}
\tau_{\mathrm{eg}}&:=\tau_A(\bT\mid\obW,\bS) + \tau_B(\bV\mid\obW,\bS)+ \tau_A(\bV\mid\obT,\bS) + \tau_B(\bT\mid\obV,\bS) + \tau_A(\bW\mid\obV,\bS)+\tau_B(\bW\mid\obT,\bS)\\
&=\tau_A(\bT\mid\obW,\bS) + \tau_B(\bT\mid\obW,\bS)+ \tau_A(\bT\mid\obV,\bS) + \tau_B(\bT\mid\obV,\bS) + \tau_A(\bW\mid\obT,\bS)+\tau_B(\bW\mid\obT,\bS),
\end{align*}
which by (\ref{eq:tau-eg}) is at most $3\tau_0 + 6\Ruz{\bX;\bY} + (I_2-I_1)$. Then observe that 
\begin{align*}
&\pot{(\bT;\bV)\mid \obW,\bS}+\pot{(\bV;\bW)\mid \obT,\bS}+\pot{(\bW;\bT)\mid\obV,\bS}\\
&= \Ruz{(\bT;\bV)\mid \obW,\bS}+\Ruz{(\bV;\bW)\mid \obT,\bS}+\Ruz{(\bW;\bT)\mid\obV,\bS}+\eta \tau_{\mathrm{eg}}\\
&\le 3I_1+6I_2 + \eta(3\tau_0 +  6\Ruz{\bX;\bY} + (I_2-I_1))\textrm{ (by (\ref{eq:endgame}))}\\
&=  24\eta\Ruz{\bX;\bY}+3\eta\tau_0 -(3+\eta)(2\eta\Ruz{\bX;\bY}-I_1) + (6+\eta)(I_2-2\eta\Ruz{\bX;\bY})\\
&= 24\eta\Ruz{\bX;\bY}+3\eta\tau_0 -\left((3+\eta)- \frac{6\eta+\eta^2}{1-\eta}\right)(2\eta\Ruz{\bX;\bY}-I_1) 
\textrm{           (by (\ref{eq:I2}))}\\
&< 3\pot{\bX;\bY}. \textrm{ (by $\eta<1/8$ and $2\eta\Ruz{\bX;\bY}-I_1\ge 0$)}
\end{align*}
Therefore either $\pot{(\bT;\bV)\mid \obW,\bS}$, $\pot{(\bV;\bW)\mid \obT,\bS}$ or $\pot{(\bW;\bT)\mid\obV,\bS}$ is less $\pot{\bX;\bY}$. If $\pot{(\bT;\bV)\mid \obW,\bS}<\pot{\bX;\bY}$ then we can take $(\bX',\bY')=(\bT|_{\obW=w,\bS=s},\bV|_{\obW=w,\bS=s})$ for $w,s$ to be the fixing that minimizes $\pot{\bT|_{\obW=w,\bS=s};\bV|_{\obW=w,\bS=s}}$, which should be at most the average $\pot{(\bT;\bV)\mid \obW,\bS}<\pot{\bX;\bY}$. Similar argument also works for the other two cases.
\end{proof}
\section{Improved Bound for Covering Variant}
In this section we prove \Cref{conj:Marton}. First of all, recall the following argument which can be found in the proof of (\Cref{conj:Marton-entropic}$\Rightarrow$\Cref{conj:Marton}) in \cite{GMT23,GGMT23}. We restate the proof for completeness.
\begin{lemma}\label{lemma:slice-to-cover}
For $A\subseteq \bbF_2^n$ such that $\abs{A+A}\le K\abs{A}$, if there exists a linear subspace $V$ and $t\in\bbF_2^n$ such that $\abs{A\cap (V+t)}\ge\max (\abs{A},\abs{V})/R$. Then $A$ can be covered by at most $2KR$ translates of a subspace $V'$ of size at most $\abs{A}$.
\end{lemma}
\begin{proof}
Let $B=A\cap (V+t)$. By Ruzsa's covering lemma~\cite[Lemma 2.14]{TV06}, $A$ can be covered by at most $\abs{A+B}/\abs{B}$ translates of $B-B\subseteq V$. Then observe that $\abs{A+B}/\abs{B}\le \abs{A+A}/\abs{B}\le K\abs{A}/\abs{B}\le KR\cdot \min(1,\abs{A}/\abs{V})$. If $\abs{V}\le\abs{A}$ then we are done with $V'=V$. Otherwise, there exists a subspace $V'$ of size at most $\abs{A}$ such that $V$ can be covered by at most $2\abs{V}/\abs{A}$ translates of $V'$. In this case $A$ can also be covered by at most $(2\abs{V}/\abs{A})\cdot KR(\abs{A}/\abs{V})=2KR$ translates of $V'$.
\end{proof}

The implication (\Cref{conj:Marton-entropic}$\Rightarrow$\Cref{conj:Marton}) comes from the fact that $\Ruz{\bU_A;\bU_V}\le r$ implies the existence of $t$ such that $\abs{A\cap (V+t)}\ge\max (\abs{A},\abs{V})/2^{2r}$. Our further improved bound in the covering variant comes from an alternative choice of $\tau_A=\tau_B$ such that $\tau_A(\bU_A)=0$, but at the same time a bound on $\tau_A(\bU_V)\le r$ is still strong enough to imply $\abs{A\cap (V+t)}\ge\max (\abs{A},\abs{V})/2^{2r}$. Therefore, compared to the choice $\tau_A(\cdot)=\Ruz{\bU_A;\cdot}$ where $\tau_A(\bU_A)\le\log(K)$ we save a factor of $K^2$.

Now fix a set $A$, and define
\begin{align*}
\tau^-\paren{\bX}&:=\inf_{\bT}\KL{\bX}{\bU_A+\bT},\\
\tau^+\paren{\bX}&:=\tau^-\paren{\bX}+\ent{\bX}-\ent{\bU_A}.
\end{align*}
Our choice of $\tau_A,\tau_B$ when applying \Cref{lemma:main} would be $\frac{1}{2}(\tau^++\tau^-)$. Note that both $\tau^+\paren{\bU_A}$ and $\tau^-\paren{\bU_A}$ are $0$, with the choice of minimizer being $\bT=0$. Before we prove that $\tau_A$ indeed satisfies our claim, first we need some simple properties for Kullback-Leibler divergence.
\begin{lemma}\label{lemma:main-eq}
Let $\bX,\bY,\bZ$ be independent random variables on $\bbF_2^n$, and $\bW$ be any random variable correlates with $\bX$. Then
\begin{itemize}
\item 
$\KL{\bX+\bZ}{\bY+\bZ}\le \KL{\bX}{\bY}$.
\item 
$\KL{\paren{\bX\mid \bW}}{\bY}= \KL{\bX}{\bY} + \ent{\bX}-\ent{\bX\mid\bW}$.
\end{itemize}
\end{lemma}
\begin{proof}
The first inequality is by convexity and the fact that $\KL{\bX+z}{\bY+z}=\KL{\bX}{\bY}$. The second equality is by the fact that $\ent{\bX:\bY}=\ent{\bX|\bW:\bY},$
where $\ent{\bX:\bY}$ denotes the cross-entropy $\ex[x\sim\bX]{-\log(\pr{\bY=x})}=\KL{\bX}{\bY}+\ent{\bX}$.
\end{proof}
\noindent
Now we prove that $\tau^-\paren{\bU_V}$ exactly captures the density of $A\cap (V+t)$.
\begin{lemma}
$\tau^-\paren{\bU_V}=\log(|A|)-\log(\max_t\abs{A\cap (V+t)})$.
\end{lemma}
\begin{proof}
First note that by the first property in \Cref{lemma:main-eq}, for any choice of $\bT$, $\KL{\bU_V}{\bU_A+\bT}\ge\KL{\bU_V}{\bU_A+\bT+\bU_V}$. Then observe that $$\KL{\bU_V}{\bU_A+\bT+\bU_V}=-\sum_{v\in V}\frac{1}{\abs{V}}\log\left(\abs{V}\Pr[\bU_A+\bT+\bU_V=v]\right)=-\log(\Pr[\bU_A+\bT\in V]).$$ Now define $t^*=\arg\max_{t}\abs{A\cap (V+t)}$. For any $\bT$, $$\KL{\bU_V}{\bU_A+\bT}\ge-\log(\Pr[\bU_A\in V+\bT])\ge -\log\left(\frac{\abs{A\cap(V+t^*)}}{\abs{A}}\right)=\KL{\bU_V}{\bU_A+t^*}.$$ Therefore the infimum of $\KL{\bU_V}{\bU_A+\bT}$ is $$\tau^-(\bU_V)=\KL{\bU_V}{\bU_A+t^*}=\log(|A|)-\log(\max_t\abs{A\cap (V+t)}).$$
\end{proof}
The lemma above directly implies that $\tau^+\paren{\bU_V}=\log(|V|)-\log(\max_t\abs{A\cap (V+t)})$. In addition, because both $\tau^+\paren{\bU_V},\tau^-\paren{\bU_V}$ are non-negative, $\max(\tau^+\paren{\bU_V},\tau^-\paren{\bU_V})\le \tau^+\paren{\bU_V}+ \tau^-\paren{\bU_V}$. Therefore we have the following claim. 
\begin{claim}\label{lemma:slice}
If $\tau^+(\bU_V)+\tau^-\paren{\bU_V}\le r$, then there exists $t$ such that $\abs{A\cap (V+t)}\ge 2^{-r}\max(\abs{A},\abs{V})$.
\end{claim}

It remains prove that $\frac{1}{2}(\tau^++\tau^-)$ satisfies the conditions in \Cref{lemma:main}. The third and fourth conditions are easy to verify. For the first two conditions, it suffices to prove the following lemma.
\begin{lemma}\label{lemma:KL-eq}
For any independent random variables $\bX,\bY$ on $\bbF_2^n$ and $\bZ$ correlated with $\bX$,
\begin{itemize}
\item
$\tau^-(\bX+\bY)\le \tau^-(\bX)$,
\item 
$\tau^-(\bX\mid\bZ)\le \tau^-(\bX)+\ent{\bX}-\ent{\bX\mid\bZ}$,
\item 
$\tau^+(\bX+\bY)\le \tau^+(\bX)+\ent{\bX+\bY}-\ent{\bX}$,
\item 
$\tau^+(\bX\mid\bZ)\le \tau^+(\bX)$.
\end{itemize}
\end{lemma}
\begin{proof}
For the first inequality, by \Cref{lemma:main-eq} we have
$$\tau^-(\bX+\bY)\le \inf_{\bT} \KL{\bX+\bY}{\bU_A+(\bT+\bY)} \le \inf_{\bT} \KL{\bX}{\bU_A+\bT} = \tau^-(\bX).$$
For the second inequality, again by \Cref{lemma:main-eq}, 
\begin{align*}
\tau^-(\bX\mid \bZ)
&= \ex[z\sim\bZ]{\inf_{\bT_z}\left(\KL{(\bX|_{\bZ=z})}{\bU_A+\bT_z}\right)}\\
&\le \inf_{\bT}\left(\ex[z\sim\bZ]{\KL{(\bX|_{\bZ=z})}{\bU_A+\bT}}\right)\\
&= \inf_{\bT}\left(\KL{(\bX\mid\bZ)}{\bU_A+\bT}\right)\\
&= \tau^-(\bX)+ \ent{\bX}-\ent{\bX\mid\bZ}.
\end{align*}
The remaining two conditions are direct corollaries of the first two.
\end{proof}
Now we are ready to prove \Cref{theorem:Marton}.
\begin{proof}[Proof of \Cref{theorem:Marton}]
Take $\tau_A=\tau_B=\frac{1}{2}(\tau^-+\tau^+)$ in \Cref{lemma:main} (which is legal by \Cref{lemma:KL-eq}). Then there exists a subspace $V$ such that $\tau^-(\bU_V)+\tau^+(\bU_V) \le \tau^-(\bU_A)+\tau^+(\bU_A) + 8\Ruz{\bU_A;\bU_A}\le 8\log(K)$, where the last inequality is by the fact that $\Ruz{\bU_A;\bU_A}\le \log(\abs{A+A})-\ent{\bU_A}\le \log(K)$. By \Cref{lemma:slice}, there exists $t$ such that $\abs{A\cap(V+t)}\ge \max(\abs{A},\abs{V})/K^8$. By \Cref{lemma:slice-to-cover}, $A$ can be covered by at most $2K^9$ translates of a subspace of size at most $\abs{A}$, which concludes the proof.
\end{proof}

\paragraph{Acknowledgements.} Thank Tim Gowers, Ben Green, Freddie Manners and Terry Tao for encouraging the author to write this paper and providing many helpful comments. We also want to thank the PFR Lean project team (\url{https://teorth.github.io/pfr/}) for formally verifying \Cref{theorem:Marton-entropic}.

\appendix
\section{Basic Information Measures}\label{appendix:definition}
In this section we briefly review some basic information measures and some of their properties that we use in this paper. Their proofs are either straightforward or can be found in textbooks of information theory, e.g. \cite{Cover}.

\paragraph{Shannon Entropy.} Let $\bX$ be any random variable on a finite set. The (Shannon) entropy of $\bX$ is defined as $\ent{\bX}:=\sum_{x\in \Supp(\bX)} -\pr{\bX=x}\log(\pr{\bX=x})$. Note that $\mathrm{H}$ only depends on the distribution of $\bX$, so we also treat it as a function on distributions. We will use the following simple properties.
\begin{itemize}
\item 
$\ent{\bX,\bY}\le \ent{\bX}+\ent{\bY}$, and $\ent{\bX,\bY}= \ent{\bX}+\ent{\bY}$ iff $\bX,\bY$ are independent.
\item
$\ent{\bX}\le\log(\Supp(\bX))$.
\item 
$\ent{\bX}=\ent{f(\bX)}$ for any $f$ that is injective on $\Supp(\bX)$.
\item
For distribution $P,Q$ over a finite set $S$ and any $\lambda\in[0,1]$, $\ent{\lambda P + (1-\lambda) Q}\ge \lambda \ent{P} + (1-\lambda) \ent{Q}$. In other words, $\mathrm{H}$ is concave.
\end{itemize}

\paragraph{Conditional Entropy.} For joint random variables $(\bX,\bZ)$ we define the conditional entropy $\ent{\bX\mid\bZ}$ to be $\ex[z\sim\bZ]{\ent{\bX|_{\bZ=z}}}$. We need the following properties.
\begin{itemize}
\item 
$\ent{\bX\mid\bZ}=\ent{\bX,\bZ}-\ent{\bZ}$
\item 
$\ent{\bX\mid\bZ}\le\ent{\bX}$
\item 
$\ent{f(\bX,\bZ)\mid g(\bZ)}=\ent{\bX\mid\bZ}$ for any $g$ injective on $\Supp(\bZ)$ and any $f(\cdot,z)$ that is injective on $\Supp(\bX)$ for every $z\in\Supp(\bZ)$. 
\end{itemize}

\paragraph{Mutual Information.} The mutual information between two correlated random variables $(\bX,\bY)$ is defined as $\info{\bX:\bY}:=\ent{\bX}+\ent{\bY}-\ent{\bX,\bY}$. For $(\bX,\bY,\bZ)$ we also define the conditional mutual entropy $\info{\bX:\bY\mid\bZ}:=\ex[z\sim\bZ]{\info{\bX|_{\bZ=z},\bY|_{\bZ=z}}}$. We have the following properties.
\begin{itemize}
\item 
$\info{\bX:\bY}\ge 0$, and $\info{\bX:\bY}=0$ iff $\bX,\bY$ are independent.
\item 
$\info{\bX:\bY}=\info{\bY:\bX}$
\item 
$\info{\bX:\bY}=\ent{\bX}-\ent{\bX\mid\bY}$. Similarly $\info{\bX:\bY\mid\bZ}=\ent{\bX\mid\bZ}-\ent{\bX\mid\bY,\bZ}$.
\item 
$\info{f(\bX,\bZ):\bY\mid g(\bZ)}=\info{\bX:\bY\mid\bZ}$ for any $g$ injective on $\Supp(\bZ)$ and any $f(\cdot,z)$ that is injective on $\Supp(\bX)$ for every $z\in\Supp(\bZ)$. 
\end{itemize}

\paragraph{Kullback-Leibler Divergence.} For distributions $P,Q$ over a finite $S$ s.t. $\Supp(P)\subseteq \Supp(Q)$, the Kullback-Leibler divergence is defined as $\KL{P}{Q}=\sum_{x\in \Supp(P)}P(x)\log(P(x)/Q(x))$. In addition, for $\Supp(P)\subsetneq \Supp(Q)$ define $\KL{P}{Q}=\infty$. We also abuse notation and define $\KL{\bX}{\bY}=\KL{P_\bX}{P_\bY}$ for two random variables $\bX,\bY$ where $P_{\bX},P_{\bY}$ denote their distributions. In addition, we write $\KL{\bX\mid\bZ}{\bY}=\ex[z\sim\bZ]{\KL{\bX|_{\bZ=z}}{\bY}}$. For Kullback-Leibler divergence we have the following properties. 
\begin{itemize}
\item 
$\KL{P}{Q}\ge 0$, and $\KL{P}{Q}=0$ iff $P=Q$.
\item 
For distributions $P_1,Q_1,P_2,Q_2$ over a finite set $S$ and any $\lambda\in[0,1]$, 
$$\KL{\lambda P_1+(1-\lambda)P_2}{\lambda Q_1 + (1-\lambda) Q_2} \le \lambda \KL{P_1}{Q_1} + (1-\lambda)\KL{P_2}{Q_2}.$$
In other words, $D_\mathrm{KL}$ is convex.
\item 
For any random variables $\bX,\bY$ on $S$ and any injection $f$, $\KL{f(\bX)}{f(\bY)}=\KL{\bX}{\bY}$.
\end{itemize}

\bibliography{AC}
\bibliographystyle{alpha}

\end{document}